\documentclass[12pt,draft]{article}
\title{	Holomorphic functions on complex Banach lattices}

\author{Christopher Boyd, Raymond A. Ryan and Nina Snigireva 
	}

\date{}

\def\cqd{\ \ \ \hbox{}\nolinebreak\hfill $\blacksquare \  \  \
\  $ \par{}\medskip}

\newenvironment{proof}{\noindent{\bf Proof:}}{\hfill \cqd}

\makeatletter
        \def\@thefnmark{\null}
        \def\footnotetexta{\@footnotetext}

\usepackage{amsmath,amsfonts,amssymb}

\usepackage{xcolor}

\usepackage{euscript}
\newtheorem{proposition}{Proposition}
\newtheorem{theorem}{Theorem}

\newtheorem{lemma}{Lemma}
\newtheorem{definition}{Definition}
\newtheorem{example}{Example}

\newcommand{\rlinn}[3]{\mathcal{L}_r(^{#1}{#2};{#3})}
\newcommand{\rlinns}[3]{\mathcal{L}^s_r(^{#1}{#2};{#3})}
\newcommand{\hpoly}[3]{\mathcal{P}(^{#1} {#2};{#3})}
\newcommand{\hpolys}[2]{\mathcal{P}(^{#1} {#2})}
\newcommand{\thetapoly}[3]{\mathcal{P}_\theta(^{#1} {#2};{#3})}
\newcommand{\rhpoly}[3]{\mathcal{P}_r(^{#1} {#2};{#3})}
\newcommand{\rhpolys}[2]{\mathcal{P}_r(^{#1} {#2})}

\newcommand{\npoly}[2]{\mathcal{P}(^{#1}{#2})}

\newcommand{\rhol}[2]{\mathcal{H}_r({#1};{#2})}
\newcommand{\rbhol}[2]{\mathcal{H}_{rb}({#1};{#2})}

\newcommand{\rhols}[1]{\mathcal{H}_r({#1})}
\newcommand{\rbhols}[1]{\mathcal{H}_{rb}({#1})}

\newcommand{\mdiff}[2]{\frac{1}{#1!} \widehat{d^#1}{#2}}	
\newcommand{\bignorm}[1]{\Bigl\| {#1}\Bigr\|}
\newcommand{\multi}{\N_0^{(\N)}}		

\DeclareMathOperator{\re}{Re}
\DeclareMathOperator{\im}{Im}

\newcommand{\R}{\ensuremath{\mathbb{R}}}
\newcommand{\C}{\ensuremath{\mathbb{C}}}
\newcommand{\N}{\ensuremath{\mathbb{N}}}

\newcommand{\rad}[2]{r(#1,#2)}
\newcommand{\radreg}[2]{|r|(#1,#2)}
\newcommand{\polyd}[2]{D(#1,#2)}

\begin{document}
\renewcommand{\baselinestretch}{1.25}
\maketitle

\begin{center}
	{\it Dedicated to Se\'an Dineen (1944--2024), teacher,
	mentor and friend.}
\end{center}

\begin{abstract}\footnotetexta{{\bf Keywords:} Holomorphic function; Homogeneous polynomial; Banach lattice; Regular polynomial; Bohr radius.}
\footnotetexta{{\bf MSC(2020):} Primary 46G20; 46B42; 32A70; Secondary 46E10; 32A05.}
\noindent
 We introduce and study the algebraic, analytic and lattice properties of regular
homogeneous polynomials and holomorphic functions on complex Banach lattices.  We show that the theory of power series with regular
terms is closer to the theory of functions of several complex variables than
the theory of holomorphic functions on Banach spaces.  We extend the concept of the Bohr
radius to Banach lattices and show that it   
provides us with a
lower bound for the ratio between the radius of regular convergence and the
radius of convergence of a regular holomorphic function. This allows us to
show that in finite dimensions the radius of convergence of the Taylor
series of a holomorphic function coincides with the radius of convergence of its
monomial expansion but that on $\ell_p$ these two radii can be radically
different.
\end{abstract}
 
\section{Introduction}

The aim of this paper is to build a framework
for the study of holomorphic functions on 
complex Banach lattices that takes  account
of the lattice structure of the domain, 
a feature that is 
rarely acknowledged in the literature in this area
at present.  The crucial concept is regularity --- 
we require that the homogeneous polynomials that 
make up the Taylor expansion of a holomorphic function
are regular, meaning that they have a modulus which is
also a homogeneous polynomial, and that the convergence of the series is regular, in a sense that we will define.  This leads to a theory
which is closer in many respects to the classical 
theory for several complex variables.  For example,
it is possible to define the concept of logarithmic
convexity for sets in a complex Banach lattice and
we can show that the domain of convergence of a power
series with regular terms has this property.

To date, regular polynomials on Riesz spaces and Banach lattices have largely
been studied for real spaces. The study of regular holomorphic functions means
that we now have to consider homogeneous polynomial on complex Banach lattices.
A complex Banach lattice is, 
by definition,
 the complexification of a real Banach lattice. Therefore, in Section 2 we will concentrate on the complexification process
and see how regular polynomials on a complex Banach lattice have real and imaginary parts which are the complexifications of regular polynomials on the associated `real part'.
Moreover, the space of regular $m$-homogeneous polynomials on a complex Banach
lattice can be given a norm, known as the regular norm, with respect to which
it becomes a complex Banach lattice.

In Section 3 we will initiate our study of power series of regular polynomials
and introduce the concept of regular convergence. 
We will see that the natural domain of regular convergence of
such a power series is a logarithmically 
convex solid domain. In Section~4 we introduce
regular holomorphic functions as those holomorphic functions
whose derivatives at every point are regular
and whose Taylor series at every point is
regularly convergent in some neighbourhood of the point.  
We shall see that the space of regular holomorphic functions forms a
holomorphy type in the sense of Nachbin, \cite{Nachbin}, and that the theory
of regular holomorphic functions is more closely related to the theory of
several complex variables than it is to the theory of holomorphic functions on
Banach spaces.

With each regular holomorphic function $f$ and each point $a$ in the domain of
$f$ we associate two radii of convergence, the radius of convergence of $f$
itself about $a$ and the radius of regular convergence. In general, the radius of
regular convergence is smaller than the radius of convergence.
In Section~5
we begin by observing how the idea of the homogeneous Bohr radii introduced by
Defant, Garc{\'\i}a and Maestre, \cite{DGM2003}, for homogeneous polynomials on
$\mathbb{C}^k$ can be extended to homogeneous polynomials on general Banach
lattices. From this viewpoint, the Bohr
radius can be seen as a measure of the 
equivalence of the regular and supremum norms.
We will then see how these Bohr radii provide us with a lower bound
for
the ratio between the radius of regular convergence and the radius of 
convergence. This will allow us to show that for finite dimensional spaces the
radii of convergence and regular convergence coincide.

In Section~6 we will show that on $\ell_p$, for $1<p<\infty$, it is possible to
construct regular holomorphic functions with radius of convergence equal to $1$
yet having arbitrary small radius of regular convergence. Finally, in
Section~7 we consider complex 
orthogonally additive polynomials and 
holomorphic functions.  
We prove that for complex orthogonally
additive polynomials the regular and 
supremum norms coincide and we use
this to show that the radius of convergence is equal to  the radius
of regular convergence on the space of orthogonally additive holomorphic
functions.

For further reading on the theory of holomorphic functions on Banach spaces we
refer the reader to \cite{Dineen1} and \cite{Dineen2} while for the theory of Banach lattices we refer the reader to \cite{AB}, \cite{MN}
and \cite{Sch}.

\bigskip

\section{Regular polynomials on complex Banach lattices}

 First we recall the definition of a complex Banach lattice. We denote by $E_\C$ the complexification 
of a real vector space $E$.  Thus,
$E_\C$ is a complex vector space in
which every element $z$ can be 
expressed uniquely in the form
$z=x+iy$, where $x,y\in E$.
The real vectors $x$, $y$ are referred
to as the real and imaginary parts,
respectively, of $z$ and we write
$x=\re z$ and $y=\im z$. If $E$ is a Banach lattice, 
the modulus of $z=x+iy\in E_\C$  is the element 
of $E$ given by
\begin{equation}\label{e:def}
|z| = \sqrt{x^2+ y^2}  = 
\sup\{x\cos\theta + y\sin\theta : 0\le\theta\le 2\pi \},
\end{equation}
where 
these expressions are
 defined using the Krivine 
functional calculus \cite[Section 1.d]{LT}. A norm
is defined on $E_\C$ by
$$
\|z\| = \|\,|z|\,\|
$$ and the triple consisting of $E_\C$ with the modulus
and norm is, by definition, a 
\emph{complex Banach lattice}.  Alternatively, one may
take an axiomatic approach.  This was done by Mittelmeyer
and Wolff \cite{MW} and they showed that every complex
Banach lattice defined in their way is the 
complexification of a (real) Banach lattice as outlined
above.  We  recommend the paper by Buskes and 
Schwanke \cite{BuskesS} for an informative discussion
of the Mittelmeyer-Wolff axioms and a general treatment
of complexification of vector lattices.

Let $E,F$ be (real) Banach lattices.
An $m$-homogeneous polynomial
$P\colon E\to F$ is generated by
a unique symmetric $m$-linear mapping
$A\colon E^m\to F$,
in the sense that $P(x)= A(x,\dots,x)$ 
for every $x\in E$.
We write $P=\widehat{A}$.  The 
polynomial $P$ is said to be
\emph{positive} if 
$A(x_1,\dots,x_m)\ge 0$
for all $x_1,\dots,x_m \ge 0$
and \emph{regular} if it is the 
difference of two positive
polynomials.  If $F$ is Dedekind
complete, then the space 
$\rhpoly{m}{E}{F}$ of 
regular $m$-homogeneous polynomials
is a Banach lattice with the 
\emph{regular norm}, defined by
$\|P\|_r
= \| |P| \|$, where $\|\cdot\|$ is
the supremum norm \cite{BB}.

Let $E_1,\dots,E_m,F$ be real vector
spaces
and let
$A\colon E_1\times \dots \times E_m \to F$
be an $m$-linear mapping.
Then
$A$ 
has a unique extension to a
complex $m$-linear mapping 
$A_\C \colon (E_1)_\C\times \dots \times (E_m)_\C
 \to F_\C $ \cite[Theorem 3]{BS71}.
 For $z_j = x^0_j+ix^1_j \in (E_j)_\C$,
 $1\le j\le m$, we have
 $$
 A_\C (z_1,\dots,z_m)=
 \sum_{\delta_1,\dots,\delta_m=0,1}
 i^{\sum \delta_j}
A(x^{\delta_1}_1, \dots, x^{\delta_m}_m)\,.
$$
We shall say that a complex $m$-linear
mapping from $(E_1)_\C\times \dots (E_m)_\C$ into
$F_\C $ is \textit{real}
if it is the complexification of
a real $m$ linear mapping from
$E_1\times \dots \times E_m$
into $F$.   Every complex $m$-linear
mapping $A$ can be decomposed 
uniquely in the form
$A=A_0 +i A_1$, where
$A_0,A_1$ are real $m$-linear mappings.
We let
$
A_0(x_1,\dots,x_m) = \re A(x_1,\dots,x_m)
$ and $A_1(x_1,\dots,x_m)= \im A(x_1,\dots,x_m)$
for real arguments and then extend $A_0,A_1$
to the complexification as described above.
It follows that the vector space of
complex $m$-linear mappings is the 
complexification of the space of 
real $m$-linear mappings:
$$
L(^m (E_1)_\C, \dots, (E_m)_\C; F_\C) \cong
L(^m E_1,\dots, E_m; F)_\C\,.
$$

If the spaces $E_1,\dots,E_m$ are the same
then clearly an $m$-linear mapping $A$
on $E^m$ is symmetric if and
only if its complexification, $A_\C$,
is.
Thus, if $E$, $F$ are real vector spaces,
then every $m$-homogeneous polynomial
$P= \widehat{A}\colon E\to F$ has a unique
extension to a complex $m$-homogeneous
polynomial $P_\C \colon E_\C \to F_\C$.
Moreover, we have
$$
P_\C(z)=
\sum_{\ell=0}^{\lfloor m/2\rfloor}
\binom{m}{2\ell}(-1)^\ell A(x^{m-2\ell},y^{2\ell})
+i \Bigl(\sum_{\ell=0}^{\lfloor m/2\rfloor} 
\binom{m}{2\ell+1}(-1)^\ell A(x^{m-2\ell-1},y^{2\ell+1})   \Bigr)
$$
for every $z= x+iy\in E_\C$.
As these complexifications are canonical,
we usually omit the subscript $\C$,
so $P(z)$ will be understood to mean 
$P_\C(z)$.
A complex $m$-homogeneous polynomial 
that arises in this way, as the complexification
of a polynomial from $E$ into $F$, will be
referred to as a 
\textit{real polynomial}.  
It follows from our remarks
above that the space of complex
$m$-homogeneous polynomials 
is the complexification of the space 
of real $m$-homogeneous polynomials:
$$
P(^m E_\C; F_\C) \cong
P(^m E;F)_\C\,.
$$ 
Every complex $m$-homogeneous polynomial $P$
can be written uniquely in the form
$P=P_0 +i P_1$, where $P_0,P_1$ are real 
$m$-homogeneous polynomials, defined,
for real arguments by
\begin{equation}\label{eq:RealIm}
P_0(x)= \re P(x) \quad\text{and}\quad
P_1(x) = \im P(x)
\end{equation}
for $x\in E$.  The polynomials $P_0, P_1$ are then extended 
by complexification to all of 
$E_\C$.  We note that it is not
correct so say that 
$P_0(z) = \re P(z)$ for 
complex vectors $z$;
the identity $(\re P)(z) =
\re(P(z))$ is only valid
for real arguments.
To give a simple example,
consider the $2$-homogeneous
polynomial $P(z)= z^2$ on $\C$.
We have $P(z)=(x^2-y^2) +i(2xy)$
for $z= x+iy$.  However,
$x^2-y^2$ and $2xy$ are not
the real and imaginary parts of $P$.
Indeed, $P$ is a real polynomial, being 
the complexification of the real 
polynomial $P(t)=t^2$ and so its
imaginary part is zero.

We now consider polynomials on Banach 
lattices.  Let $E$ be a complex Banach
lattice.  The real part of $E$,
denoted by $E_\R$, is defined to be
the real linear span of the set
$\{|z|: z\in E\}$
and is a (real) Banach lattice,
with the norm induced from $E$.
Then $E$, as a vector space,
is the complexification of $E_\R$. We recall that the norms
on $E$ and $E_\R$ are connected by the 
relation $\|z\|_E = \| |z| \|_{E_\R}$ for $z\in E$. 
Our definition of regularity, positivity, respectively, for multilinear
and polynomial mappings is the same
as the linear case \cite[Section 2.2]{MN}.
So we say that an $m$-linear mapping
on $E^m$,
or an $m$-homogeneous polynomial on $E$,
is regular, positive, respectively, if both its real
and imaginary parts are regular, positive on $(E_\R)^m$
or $E_\R$, respectively.

If
 $E, F$ are (real) Banach lattices with $F$ Dedekind 
complete, then the space $\rhpoly{m}{E}{F}$
of regular $m$-homogeneous polynomials, with
the regular norm $\|P\|_r = \| |P|\|$,
is a Banach lattice \cite{BB}.    
Now let  $E, F$ be complex 
Banach lattices, with $F$ Dedekind complete.
Then it follows from the above that
the space $\rhpoly{m}{E}{F}$
of regular $m$-homogeneous polynomials
is the complexification 
of the space $\rhpoly{m}{E_\R}{F_\R}$
of real, regular $m$-homogeneous
polynomials:
$$
\rhpoly{m}{E}{F}=
\bigl(\rhpoly{m}{E_\R}{F_\R}\bigr)_\C
\,. $$
Therefore, $\rhpoly{m}{E}{F}$ 
can be endowed with a complex Banach lattice structure, 
with the norm given by  $\|P\|_r = \|\,|P|\,\|$.
We summarize these observations:

\begin{proposition}
	Let $E$, $F$ be complex Banach lattices, with $F$ 
	Dedekind complete.  
	The space of regular $m$-homogeneous
		polynomials from $E$ into $F$
		with the regular norm $\|P\|_r = \|\,|P|\,\|$ is
		a Dedekind complete complex Banach lattice.
\end{proposition}

Note that this result in particular tells us that the absolute
value of a complex regular
$m$-homogeneous polynomial 
$P=P_0 +i P_1$ satisfies
$$
|P| = \sqrt{|P_0|^2 + |P_1|^2}\,,
$$
 where this formula is understood in the sense of the Krivine functional calculus \cite{LT}.

 To illustrate the difference between
 the regular and supremum norms,
 we have the following useful
 result by Choi-Kim-Ki \cite{CKK98}:
 
 \begin{proposition}[{\cite{CKK98}}, Theorem 2.4]
 	Let $a,b,c\in \R$ with $|a|<1$, $|b|<1$
 	and $2<|c|\le 4$.
 	Suppose $P(x,y)= ax^2+by^2+cxy \in  \hpolys{2}{\ell_1^2} $.
 	Then, over both the real and  complex numbers,
 	$$
 	\|P\|=1\quad\text{  if and only if }\quad
 	4|c|-c^2 = 4(|a+b|-ab)\,.
 	$$
 \end{proposition}
 
 We can use this result to construct examples of polynomials
 on $\ell_1^2$ for which the regular
 norm is greater than the supremum norm.
 
 \begin{example}\label{norm}
 	On $\ell_1^2$, let
 	$$ 
 	P(x,y) =
 	\frac{1}{2}x^2 -\frac{1}{2}y^2 + (2+\sqrt{3}) xy\,.
 	$$
 	Then 
 	$$
 	\|P\|=1 \quad \text{ and }\quad
 	\|P\|_r = \frac{3+\sqrt{3}}{4}>1 \,.
 	$$
 	 over both 
 	 the real and complex numbers.
  \end{example}
 
 The fact that $\|P\|=1$ follows from the Choi-Kim-Ki result,
 but can also be seen by means of an elementary calculation.  We have
 $$
 |P|(x,y) = \frac{1}{2}x^2 +\frac{1}{2}y^2 + (2+\sqrt{3}) xy
 $$
 and thus 
 $$
 \|P\|_r \ge 
 |P|\Bigl(\frac{1}{2},\frac{1}{2}\Bigr) = 
 \frac{3+\sqrt{3}}{4}> 1\,.   
 $$
 A calculation  shows
 that in fact, $\|P\|_r = \dfrac{3+\sqrt{3}}{4}$ in the real case.
 As the coefficients of $|P|$ are positive, the regular norm
 over the complex numbers is the same.

The absolute value of a regular linear 
mapping $T\colon E\to F$ between complex
Banach lattices satisfies
$$
|T(z)| \le |T|(|z|)
$$
for every $z\in E$ \cite{MN}.
We will show that there is a corresponding result
for complex regular polynomials.
For this, we need a little 
preparation.

Let $E, F$ be real Banach lattices,
with $F$ Dedekind complete.
The Banach lattice of
regular $m$-linear mappings
from $E$ into $F$ is denoted by
$\rlinn{m}{E}{F}$
and the Banach sublattice of
symmetric, regular $m$-linear mappings
is denoted by $\rlinns{m}{E}{F}$.
The definition of the 
order relation shows that 
$\rhpoly{m}{E}{F}$ and
$\rlinns{m}{E}{F}$
are isomorphic as vector lattices.
Bu and Buskes \cite{BB} showed that
for $A\in \rlinn{m}{E}{F}$,
the absolute value $|A|$ is
given by
	$$
|A|(x_1,\dots,x_m) = 
\sup\Bigl\{\, \sum_{i_1}\dots\sum_{i_m}
|A(u^1_{i_1},\dots,u^m_{i_m})|: u^1 \in \Pi(x_1),\dots, u^m\in \Pi(x_m)\Bigr\}
$$
for $x_1,\dots,x_m \ge 0$.
Here, $\Pi(x)$ denotes the set
of partitions of a positive element
$x$, namely, finite sets of positive
vectors whose sum is $x$.
Now $\Pi(x)$ is directed by set 
inclusion and so the iterated suprema
that appear in the formula above
can be interpreted as limits
of increasing nets, with
$\Pi(x_1),\dots,\Pi(x_m)$
as the indexing sets.  Thus, we 
can interchange the operations 
of supremum with the finite sums.

If we define
$$
A^{(k)}\colon E^{m-k} \to
\rlinn{k}{E}{F}
$$
by
$$
A^{(k)}(x_1,\dots,x_{m-k})
(x_{m-k+1},\dots,x_m)
=A(x_1,\dots,x_m)\,,
$$
then the correspondence
$A\leftrightarrow A^{(k)}$
is an
isometric isomorphism between
$\rlinn{m}{E}{F}$
and $\rlinn{m-k}{E}{\rlinn{k}{E}{F}}$.
Furthermore, it follows from the remarks above
that we have $|A|^{(k)} = |A^{(k)}|$ for 
every $A\in \rlinn{m}{E}{F}$ and 
every $k$.  Therefore, this correspondence
is also a lattice homomorphism.
Thus, the Banach lattices
$\rlinn{m}{E}{F}$ 
and $\rlinn{m-k}{E}{\rlinn{k}{E}{F}}$
are isometrically lattice isomorphic.

\begin{proposition}
	\label{Iso}
	Let $E$, $F$ be real or complex Banach lattices,
	with $F$ Dedekind complete.  Then the mapping
	$$
	A\in \rlinn{m}{E}{F} \mapsto 
	A^{(k)} \in \rlinn{m-k}{E}{\rlinn{k}{E}{F}}
	$$
is a Banach lattice isometric isomorphism
for every $k= 1, \dots, m-1$.
\end{proposition}

\begin{proof}
	We have already established this result
	for real Banach lattices.
Now let $E,F$ be complex Banach
lattices, with $F$ Dedekind complete.
Then, complexifying the above isomorphism,
we have 
\begin{align*}
\rlinn{m}{E}{F} &=
\rlinn{m}{E_\R}{F_\R}_\C 
\cong\rlinn{m-k}{E_\R}{\rlinn{k}{E_\R}{F_\R}}_\C \\
&\cong
\rlinn{m-k}{E}{\rlinn{k}{E_\R}{F_\R}_\C}
\cong 
\rlinn{m-k}{E}{\rlinn{k}{E}{F}}\,.
\end{align*}
Therefore
$\rlinn{m}{E}{F}$ and 
$\rlinn{m-k}{E}{\rlinn{k}{E}{F}}$
are isometrically isomorphic as
complex Banach lattices.
\end{proof}

\begin{proposition}
	Let $E$, $F$ be complex Banach lattices, with $F$ 
	Dedekind complete and let
	$P\colon E\to F$ be a regular
	$m$-homogeneous polynomial. Then
	$$
	|P(z)| \le |P|(|z|)
	$$
	for all $z\in E$.
\end{proposition}

\begin{proof}
	Let $A$ be the associated
	regular $m$-linear mapping.
	We will show that 
	$$
	|A(z_1,\dots, z_m)|\le
	|A|(|z_1|,\dots,|z_m|)
	$$
	for all $z_1,\dots,z_m\in E$.
	The proof is by induction.
	The case $m=1$ is the linear
	result \cite{MN}.  Assume the result
	holds for $(m-1)$-linear
	mappings. Then, 
	using Proposition \ref{Iso},
	\begin{align*}
	|A(z_1,\dots, z_m)| &=
	|A^{(m-1)}(z_1)(z_2,\dots,z_m)|
		\le
	|A^{(m-1)}(z_1)|(|z_2|,\dots,|z_m|)\\
	& \le
	|A^{(m-1)}|(|z_1|)(|z_2|,\dots,|z_m|))
	=|A|(|z_1|,\dots,|z_m|)\,.
	\end{align*}
Thus the result follows for every $m$.	
\end{proof}

Our next result shows that the complexification process preserves the absolute value.
\begin{proposition}\label{mod}
	Let $E$, $F$ be real Banach lattices, with $F$ Dedekind 
	complete and let $P\in \rhpoly{m}{E}{F}$. Then
	$|P_\C| = |P|_\C$. 
\end{proposition}

\begin{proof}
	We begin with the case $m=1$.  Let $T\colon E\to F$ 
	be a regular linear operator.  For $u\in E_+$ we have
	\begin{align*}
		|T_\C|(u) &= \sup\{|T_\C(x+iy)|: x,y\in E, |x+iy|\le u\}
		 = \sup\{|Tx+iTy|, |x+iy|\le u \}\\
		 &= \sup \sup_{\theta\in \R} \{|(\cos\theta) Tx +(\sin\theta) Ty|,
		 |x+iy| \le u\}\\
		 &\le \sup \sup_{\theta\in \R} \{|T|| (\cos\theta)x +(\sin\theta)y|,
		 |x+iy| \le u\}\\
		 &\le \sup\{|T||x+iy|, |x+iy| \le u\}
		 = |T|(u)  = |T|_\C (u)\,.
	\end{align*}
On the other hand,
\begin{align*}
	|T|_\C(u) &= |T|(u) = \sup\{|T(v)|:v\in E, |v| \le u\}\\
	&\le \sup\{|T(x+iy)| : x,y\in E, |x+iy| \le u\} = |T_\C|(u)
\end{align*}
for every $u\in E_+$.  Therefore $|T_\C|=|T|_\C$ on $E_+$ and 
it follows that this holds on $E_\C$.

For the case $m>1$ we have, using 
Proposition~\ref{Iso}, Banach lattice 
isomorphisms
$$
\rlinn{m}{E}{F}_\C \cong \rlinn{m}{E_\C}{F_\C}
\cong 
\rlinn{}{E_\C}{\rlinn{m-1}{E_\C}{F_\C}} 
\cong \rlinn{}{E}{\rlinn{m-1}{E}{F}}_\C \,.
$$
Thus the result follows by induction on $m$.
\end{proof}

The supremum norm of a polynomial on real Banach spaces is not,
in general, preserved by complexification.  Gustavo, Mu\~noz, Tonge 
have shown that if $E$, $F$ are
real Banach spaces and $P\colon E\to F$ is a bounded $m$-homogeneous
polynomial, then the norm of its complexification satisfies
$$
\|P_\C\|_\nu \le 2^{m-1}\| P\|
$$
where $\nu$ is any natural complexification process
\cite[Prop.~18]{MST}, and this inequality is sharp.
In the Banach lattice case, the regular norm
is much better behaved.

\begin{theorem}\label{regnorm}
		Let $E$, $F$ be real Banach lattices, with $F$ Dedekind 
	complete and let $P\in \rhpoly{m}{E}{F}$. Then
	$\|P_\C\|_r = \|P\|_r$.	 
\end{theorem}

\begin{proof}
	Clearly, we have $\|P\|_r \le \|P_\C\|_r$.  Conversely,
	using Proposition \ref{mod},
	\begin{align*}
		\|P_C\|_r &= \| \,|P_\C|\, \| = \| \, |P|_\C\|
		= \sup\{\bigl| |P|_\C (z)\bigr|: z\in E_\C, \|z\| \le 1\} \\
		&\le \sup\{|P|_\C (|z|): z\in E_\C, \|z\| \le 1\}
		= \sup\{|P| (|z|): z\in E_\C, \|z\| \le 1\} = \|P\|_r\, .
	\end{align*}
\end{proof}

Now let $E$, $F$ be complex
Banach lattices.  A complex linear
operator $T\colon
E \to F$ is a homomorphism
if $|T(z)| = T(|z|)$ for every $z\in E$.
This is equivalent to $T$ being the
complexification of a lattice homomorphism
between the real vector lattices $E_\R$ and $F_\R$
\cite[p.~136]{Sch}.\\[1pt]

Recall that the $k$th Fr\'echet
derivative of an $m$-homogeneous
polynomial $P=\widehat{A}\colon E\to F$
between real or complex vector spaces
is given  by
$$
\widehat{d}^k P(x)(y)=
m(m-1)\dots (m-k+1)A(x^{m-k} y^{k})\,.
$$
Thus  $\widehat{d}^k P$
is, up to a constant multiple, the $k$-homogeneous polynomial
associated with the 
$(m-k)$-linear
mapping $A^{(k)}$.
The following proposition follows immediately from this fact.

\begin{proposition}\label{DiffHom}
	Let $E$, $F$ be real or complex Banach lattices,
	with $F$ Dedekind complete.
	Then
	$$
	\bigl|\widehat{d^k}P\bigr| = \widehat{d^k}|P|
	$$
	for every $P\in \rhpoly{m}{E}{F}$
	and every $k$, $1\le k\le m$.
	Therefore, the mapping
	$$
	\widehat{d^k}\colon
	\rhpoly{m}{E}{F} \to
	\rhpoly{m-k}{E}{\rhpoly{k}{E}{F}}
	$$
	is a vector lattice isomorphism
	onto its image.
\end{proposition}

We point out that the mapping
$\widehat{d^k}$ is not in general an isometry
in either the real or complex cases.
Taking $E=\ell_1$ and $F= \R$ or $\C$,
the $2$-homogeneous polynomial
$P(x)=x_1x_2$ has regular norm
$\|P\|_r = 1/4$, while
the regular norm of the linear
mapping $\widehat{d^1}P\colon
\ell_1 \to \ell_\infty$ 
is
$\|\widehat{d^1}P\|_r = 1/2$.

\bigskip
The Krivine functional calculus \cite[1.d.1]{LT}   
allows one to prove
a number of H\"older type inequalities.
For example, we have
\begin{equation}\label{LT}
\| |x|^\theta |y|^{1-\theta}\| \le
\|x\|^\theta \|y\|^{1-\theta}
\end{equation}
for all elements $x, y$ of a real or complex
Banach lattice, with $0<\theta<1$.  Our next result will prove
useful when dealing with power series whose
terms are regular homogeneous polynomials.
We refer to Kusraev
\cite{Kusraev07}
for some similar
inequalities.

\begin{proposition}\label{p:holder}
	Let $E, F$ be real or
	complex Banach lattices,
	with $F$ Dedekind complete,  and let
	 $P\colon E \to F$ be a positive
	$m$-homogeneous polynomial.  Then
	$$
	P(|x|^\theta |y|^{1-\theta}) \le 
	P(|x|)^\theta \, P(|y|)^{1-\theta}
	$$
	for every $0<\theta<1$ and  every $x,y\in E$.
\end{proposition}

\begin{proof} 	
Let $a= |x|\vee|y|$. The ideal $E_a$,
which contains $x$ and $y$, is Banach lattice isomorphic
to  $C(K)$ for some compact, Hausdorff space $K$ and the expression $|x|^\theta|y|^{1-\theta}$
defined by the Krivine functional calculus
coincides with its value in $C(K)$.  
The restriction of $P$ to $E_a$ is positive and so,
by Fremlin's theorem \cite{Fremlin1}, there exists a regular,
positive Borel measure
$\mu$ on $K^m$ such that 
$$
P(w) = \int_{K^m} w(t_1) \dots w(t_m)\,d\mu(t_1,\dots,t_m)
$$
for every $w\in E_a$.  Let $\theta\in(0,1)$.
Applying  H\"older's inequality with $p=1/\theta$, 
we have
\begin{align*}
P(|x|^\theta |y|^{1-\theta}) &=
 \int_{K^m} |x(t_1)|^\theta\dots |x(t_m)|^\theta
|y(t_1)|^{1-\theta}\dots|y(t_m)|^{1-\theta}\,
d\mu(t_1,\dots,t_m)\\
&\le  \Bigl(\int_{K^m} |x(t_1)\dots x(t_m)| \,d\mu(t)\Bigr)^\theta
\Bigl(\int_{K^m} |y(t_1)\dots y(t_m)| \,d\mu(t)\Bigr)^{1-\theta}\\
&= P(|x|)^\theta \, P(|y|)^{1-\theta}\,.
\end{align*}
\end{proof}

Taking $P=\varphi^m$, where $\varphi\in E'$, and $F=\C$, we see that (\ref{LT})
is a consequence of this result.

\bigskip

\section{Power series with regular terms}

Absolute convergence of a power 
series $\sum_m c_m z^m$ in one complex
variable at a point requires that 
the numbers $|c_m z^m|$ are summable.
Of course, this is the same as
summability of the numbers $|c_m| |z|^m$.
However, on a Banach lattice, these 
lead to two different conditions.

We shall say that a power series $\sum_m P_m$ 
with regular terms on
a complex Banach lattice  $E$
is \emph{regularly convergent} at a point
$z\in E$ if the series $\sum_m |P_m|(|z|)$ converges.
Note that this is stronger than absolute convergence,
since
$$
|P_m(z)| \le |P_m|(|z|)
$$
for all $z\in E$.
We shall see that regular convergence
is in fact a strictly stronger
condition than absolute convergence.

\bigskip
If $E$ has a $1$-unconditional Schauder basis,
then, by \cite{GrecuR}, every regular $m$-homogeneous polynomial 
$P_m$ on $E$ has a pointwise unconditionally convergent
monomial expansion of the form
$$
P_m(z)  = \sum_{|\alpha|=m } c_\alpha z^\alpha\,.
$$
The converse is also true:  every unconditionally
pointwise convergent monomial expansion of degree $m$ defines a regular $m$-homogeneous polynomial.
Furthermore, by  \cite{GrecuR}, we have
$$
|P_m|(z) = \sum_{|\alpha|=m } |c_\alpha| z^\alpha\,.
$$
Thus a  power series $\sum_m P_m(z)$
with regular terms
can be written as a formal monomial expansion
$$
\sum_{m=0}^\infty \sum_{|\alpha|=m} c_\alpha z^\alpha
$$
and regular convergence of the power series at a point $z$
is equivalent to the condition
$$
\sum_{\alpha\in \multi } |c_\alpha| |z|^\alpha < \infty\,.
$$
So we see that regular convergence of a power
series is a natural
abstraction to the Banach lattice setting 
of the familiar concept of pointwise absolute
convergence of a monomial expansion.

\bigskip
For a power series $f=\sum_m P_m$
on a complex Banach space,
the radius of convergence is 
given by the formula
$$
\rad{f}{0} =
\Bigl( \limsup \|P_m\| ^{1/m}\Bigr)^{-1}
\,.
$$
This number is the supremum of the set
of nonnegative real numbers $r$ for 
which the series is uniformly convergent
on the closed ball of radius $r$.  Bearing in 
mind that the regular norm of
a regular $m$-homogeneous polynomial
$P_m$ is given by $\|P_m\|_r = 
\| |P_m| \|$, we may define
the \emph{radius of regular convergence} of a power series
$f=\sum_m P_m$ to be
$$
\radreg{f}{0} =
\Bigl( \limsup \|P_m\|_r ^{1/m}\Bigr)^{-1}
\,.
$$
This is the supremum of the set
of nonnegative real numbers $\rho$ for
which the series is uniformly 
regularly convergent in the closed 
ball of radius $\rho$.  
Since the regular norm is at least
as big as the uniform norm,
we see that the radius of regular
convergence is no bigger than 
the radius of convergence, i.e.,
$$
\radreg{f}{0}\le \rad{f}{0}\,.
$$
We shall see that these 
radii can be different.

\bigskip
The radius of regular convergence
gives some useful information about 
the behaviour of a power series.  
However, this information is somewhat 
limited as the natural domain
of convergence is not generally a ball.
Let us recall some facts about power
series in several complex variables.
A subset $D$ of $\C^k$ is 
a complete Reinhardt domain
if $z\in D$ implies that
$w\in D$ whenever $|w_j| \le |z_j|$
for $1\le j\le k$. 
A complete
Reinhardt domain $D$ is said to be
logarithmically convex 
if, for every $z,w\in D$
and every $\theta\in (0,1)$
the point 
$(|z_1|^\theta|w_1|^{1-\theta}, \dots,
|z_k|^\theta|w_k|^{1-\theta})$
belongs to $D$.
The domain of convergence of a 
power series on $\C^k$ is defined
to be the interior of the set
of points at which the series 
is absolutely convergent. 
In general, the domain of
convergence is either a non-empty
open set, or is empty. 
It is a fundamental result
for several complex variables that
the domain of convergence of a power
series is a logarithmically convex
Reinhardt domain \cite[Section 2.4]{Hormander}.

\bigskip
\begin{example}
	This result of Matos \cite{Matos}
	illustrates the difference in behaviour 
	of the domain of convergence in finite
	and infinite dimensions:  
	if $z=(z_j)$ is  a sequence of
	complex numbers, then the series
	$$
	\sum_{\alpha\in \multi} z^\alpha
	$$
	is absolutely convergent if and only
	if $z$ belongs to $\ell_1$
	and $|z_j|<1$ for every $j$.
	Therefore, on $\ell_1$ this
	monomial expansion has a non-empty
	domain of convergence.  However, on $c_0$,
	the set of points of absolute convergence
	is a dense subset of the closed unit ball
	with no interior points.
\end{example}

We now see how some results
from several complex variables
can be formulated for power series
on complex Banach lattices.
We recall the following elementary fact:
If $\sum c_\alpha (z-w)^\alpha$ is a monomial
expansion on $C^k$ whose terms are
bounded at some point $z=w+a$
where $a>0$,
then the expansion is absolutely
convergent on the polydisc
$\polyd{w}{a}= \{z\in C^k: |z_j-w_j| \le 
a_j, 1\le j\le n\}$ and the convergence
is uniform on smaller polydiscs
$\polyd{w}{\lambda a}$, $0<\lambda<1$.
We extend the definition of a polydisc
to complex Banach lattices in the obvious 
way.  If $a$ is a positive element
in a complex Banach lattice $E$,
the polydisc with centre $w$
and  polyradius $a$
is  $\polyd{w}{a}=\{z\in E: |z-w|\le a\}$.
We note that $\polyd{0}{a}$ is the closed
unit ball of the principal ideal $E_a$
with the order unit norm defined by $a$.
Note that if a power series 
centred at the origin is regularly
convergent at a point $z$, then it
is regularly convergent at every  
point in the polydisc $\polyd{0}{|z|}$.

\begin{lemma}
	Let $E$ be a complex Banach lattice
	and let $(P_m)$ be a sequence
	of regular $m$-homogeneous polynomials
	on $E$.
	If $\{P_m(a)\}$ is a bounded subset of
	$\C$ for some positive element $a\in  E$,
	then the power series $\sum_m P_m$
	is regularly convergent on the polydisc
	$\polyd{0}{a}$ and the convergence is 
	uniform on every smaller polydisc
	$\polyd{0}{\lambda  a}$,
	$0<\lambda <1$.
\end{lemma}

\begin{proof}
	This follows from the fact that
	if $|z| \le \lambda a$, then
	$||P_m|(z)| \le \lambda^m |P_m|(a)$.
\end{proof}
This is stronger than the corresponding result
for Banach spaces.  In general, all we can deduce
from the boundedness of the values $P_m(z)$
at some point in a Banach space is that the
power series $\sum_m P_m$ is absolutely
convergent in the one dimensional complex disc determined
by the point $z$.

We recall that a set $D$ in a real or complex
 Banach lattice is said to be solid if $z\in D$
 and $|w|\le |z|$ imply that $w$ belongs to 
 $D$.  For subsets of $\C^n$, the solid
 sets are precisely the complete Reinhardt domains.

\bigskip
A solid subset $D$ of $E$ is said to be
\emph{logarithmically convex} if, for 
every $x,y\in D$, $\theta\in(0,1)$,
the point $|x|^\theta|y|^{1-\theta}$ also
belongs to $D$. 
 The domain of convergence
of a power series in finite
dimensions is logarithmically convex.  The same is true for regularly convergent 
power series on a Banach lattice:
\begin{theorem}
	Let $E$ be a complex Banach lattice and let $\sum_m P_m$
	be a power series on $E$ with regular terms.
	Then the set of points at which the series 
	 converges regularly is solid and logarithmically convex.
\end{theorem}
\begin{proof}
	Let $z\in D$ and suppose $w\in$
	satisfies $|w| \le |z|$.  It follows
	from the fact that $|P_m|(|w|)\le |P_m|(|z|)$
	that $w\in D$.  Thus $D$ is a solid set.
	
	Now let $z, w\in D$.
Using  Proposition \ref{p:holder} and  the inequality between the weighted geometric and arithmetic
means, we get
$$
|P_m|(|z|^\theta |w|^{1-\theta }) 
\le |P_m|(|z|)^\theta  |P_m|(|w|)^{1-\theta}
\le
\theta |P_m|(|z|) + (1-\theta)|P_m|(|w|)
$$
and it follows immediately that $|z|^\theta |w|^{1-\theta }\in D$.
\end{proof}

\bigskip

\section{Regular holomorphic functions}

\bigskip
Let $E, F$ be complex Banach lattices, with $F$ Dedekind complete. Let $U$ be an open subset
of $E$.  A function $f\colon U\to F$ is
\textit{regularly holomorphic}
if 
\begin{enumerate}
	\item[(a)] $f$ is holomorphic on $U$.
	\item[(b)] For every $z\in U$, the derivatives
	$\displaystyle\mdiff{m}{f}(z)$ are regular $m$-homogeneous polynomials.
	\item[(c)] For every $z\in U$, the Taylor series of $f$ at $z$
	is regularly convergent in some neighbourhood of $z$.
\end{enumerate}
Suppose a power series $\sum P_m$
is regularly convergent at some point $z$.
This means that the series $\sum|P_m|(|z|)$
converges.  Then, since $\bigl||P_m|(z)\bigr|
\le |P_m|(|z|)$, it follows that the 
series $\sum|P_m|(z)$ is absolutely convergent.
So, applying the Cauchy-Hadamard condition
to the power series $\sum|P_m|$, 
we see that condition (c) above is equivalent
to 
\begin{equation}\label{e:radius}
\limsup_m \,\Bigl\| \mdiff{m}{f}(z)\Bigl\|_r^{\frac{1}{m}}
<\infty
\end{equation}
for every $z\in U$. We denote the space 
of regularly holomorphic functions
by $\rhol{U}{F}$,
or by $\rhols{U}$ when $F=\C$.
\color{black}

Let $U$ be an open set and let $z$ be a point in $U$.
We shall say that a function $f\colon U\to F$ is \textit{regularly holomorphic at $z$} if the derivatives 
	$\displaystyle\mdiff{m}{f}(z)$
of $f$
at $z$ are regular and (\ref{e:radius}) holds.

\begin{proposition}
	If a holomorphic mapping is regularly holomorphic
	  at a point in its domain, then it
	is regularly holomorphic in some
	neighbourhood of that point.
\end{proposition}
\begin{proof}
	Suppose that $f$ is holomorphic on
	an open set $U$ and that $f$ is regularly holomorphic
	at a point $z_0$ in $U$.
	Let $\displaystyle P_m = \widehat{A_m} = \mdiff{m}{f}(z_0)$. Then there 
	exist $C, \rho >0$ such that
	$$
	\|P_m\|_r \le C\, \rho^m
	$$
	for every $m\in \N$.
	If $\|z-z_0\|< \rho$, then the derivatives
	of $f$ at $z$ are given by
	$$
	\mdiff{k}{f}(z)=
	\sum_{m\ge k} \binom{m}{k} A_m(z-z_0)^{m-k}
	$$
	where this series converges in 
	$\hpoly{k}{E}{F}$ with the supremum
	norm.
	We will show that this series is
	absolutely convergent in $\rhpoly{k}{E}{F}$ 
	with the regular norm.
	Taking $\|z-z_0\|\le \sigma<(
	2
	e\rho)^{-1}$, we have
	\begin{align*}
	\sum_{m\ge k} & \bignorm{
		\binom{m}{k}  A_m(z-z_0)^{m-k}}_r
	=
	\sum_{m\ge k} \bignorm{
		\binom{m}{k}\bigl|A_m(z-z_0)^{m-k}
		\bigr|}\\
	&\le
	\sum_{m\ge k} \bignorm{
		\binom{m}{k}\bigl|A_m\bigr|\,
		\bigl|z-z_0|^{m-k}
		\bigr|}
	\le 
	\sum_{m\ge k}
	\binom{m}{k}\bigl\|A_m\bigr\|_r
	\bigl\|(z-z_0)\bigr\|^{m-k}\\
	&\le
	\sum_{m\ge k}
	\binom{m}{k}e^m\bigl\|P_m\bigr\|_r
	\bigl\|(z-z_0)\bigr\|^{m-k}
	\le
	C\, \sum_{m\ge k}2^m e^m\rho^m 
	\sigma^{m-k}\\
	&= 
	\frac{C}{1-2e\rho\sigma}\bigl(
	2
	e\rho\bigr)^k \,.
	\end{align*}
	Therefore $\displaystyle \mdiff{k}{f}(z)$
	is a regular $k$-homogeneous polynomial
	for every $k$, provided 
	$\|z-z_0\| < (2e\rho)^{-1}$.
	Furthermore, the above calculation
	shows that 
	$$
	\limsup_k\, \bignorm{\mdiff{k}{f}(z)}^{\frac{1}{k}}_r
	<\infty\,.
	$$
	Therefore $f$ is regularly holomorphic
	in the open ball with centre $z_0$
	and radius $(2e\rho)^{-1}$.
	
\end{proof}

We now look at the special case of Banach lattices
in which the lattice structure is defined by
an unconditional Schauder basis.
Let $E$ be complex Banach space with
an unconditional Schauder basis $(e_n)_n$. 
We may assume, without loss of generality,
that every point $z$ in $E$
with coordinate expansion $z=\sum_j z_j e_j \in E$ satisfies
$$
\|z\| = \sup\Bigl\{\Bigl\|\sum_j w_j e_j\Bigr\|: |w_j|\le |z_j| \quad \text{ for every } j\Bigr\}\,.
$$
(The norm on $E$ can always be replaced
by an equivalent norm with this property.)
Then $E$ is a complex Banach lattice
with the modulus defined coordinatewise:
$$
|z| = \sum_j |z_j| e_j.
$$

Now let $P_m= \widehat{A}_m$ be a bounded
$m$-homogeneous polynomial on $E$.
The value of $P_m$ at a point $z$ may
be expanded, using the multilinearity
of $\widehat{A}_m$ to give an expression 
such as
$$
P_m(z) = \sum_{j_m}\dots \sum_{j_1}
A_m(e_{j_1},\dots,e_{j_m})\, z_{j_1}\dots z_{j_m}
\,.
$$
In general, there is no guarantee that this
multiple series will converge absolutely.  
Matos and Nachbin \cite{MatosNachbin} isolated the space
of $m$-homogeneous polynomials for 
which this expansion is absolutely 
convergent at every point in $E$.
They defined a norm for this space
and proved that it is a Banach space 
in this norm.  Grecu and Ryan \cite{GrecuR}
showed that the Matos-Nachbin
polynomials coincide with the
polynomials that are regular
with respect to the Banach lattice
structure of $E$, and furthermore,
the Matos--Nachbin norm is precisely
the regular norm.  In other words,
the space of regular $m$-homogeneous
polynomials on $E$ is exactly the 
space of $m$-homogeneous polynomials
$P_m$ that have pointwise absolutely
convergent monomial expansion of the 
form
$$
P_m(z) = \sum_{\stackrel{\alpha\in \multi}
	{|\alpha|=m}}
c_\alpha z^\alpha,
$$
with
$$
|P_m|(z) = \sum_{\stackrel{\alpha\in \multi}
	{|\alpha|=m}}
|c_\alpha| z^\alpha.
$$
The regular norm of $P_m$ is
$$
\|P_m\|_r = \sup\Bigl\{ \Bigl\|\sum_{|\alpha|=m}|c_\alpha||z|^\alpha\Bigr\| : \|z\| \le 1\Bigr\} \,.
$$

\bigskip

The following result, and its proof, is
based on Theorem 3.10 in \cite{Matos}.

\begin{theorem}[cf.~\cite{Matos}, Theorem~3.10]
	Let $E$ be a complex Banach 
	space with an unconditional Schauder
	basis
	and let $U$ be an open subset of $E$.
	A function $f\colon U \to \C$ is
	regularly holomorphic if and only if,
	for every $z\in U$, the monomial
	expansion of $f$ around $z$  is absolutely
	convergent to $f$ in some neighbourhood
	of $z$.
\end{theorem}

\begin{proof}
	
	First, suppose that $f$ is regularly
	holomorphic on $U$.  Let $z\in U$. 
	Then 
	$$
	\limsup_m \,\Bigl\| \mdiff{m}{f}(z)\Bigl\|_r^{\frac{1}{m}}
	=\rho
	<\infty
	$$
	and so there is a positive constant $C$
	satisfying 
	$$
	\Bigl\| \mdiff{m}{f}(z)\Bigl\|_r
	\le C\, \rho^m
	$$
	for every $m\in \N$.
	Now $\mdiff{m}{f}(z)$,
	being regular, has an absolutely
	pointwise convergent monomial 
	expansion of the form \cite{Matos}
	$$
	\mdiff{m}{f}(z)(w) =
	\sum_{|\alpha|=m} c_\alpha w^\alpha
	$$ 
	and so we have
	$$
	\sum_{|\alpha|=m}|c_\alpha|\,|w-z|^\alpha
	\le \Bigl\| \mdiff{m}{f}(z)\Bigl\|_r
	\|w-z\|^m
	\le C\,\rho^m \,\|w-z\|^m\ \,.
	$$
	It follows that the monomial
	expansion $$
	\sum_{\alpha\in \multi} c_\alpha (w-z)^\alpha
	$$
	is absolutely convergent to $f(w)$ for 
	$\|w-z\|< \rho^{-1}$.
	
	Now suppose that $f$ is locally representable
	by pointwise absolutely convergent 
	monomial expansions.  So, for each $z\in U$,
	there exists $\rho>0$
	and a monomial expansion of the form
	$$
	f(w)= \sum_{\alpha\in \multi} c_\alpha (w-z)^\alpha
	$$
	that converges absolutely at every 
	point in the ball $\|w-z\|<\rho$.
	It follows that $f$ is Gateaux-holomorphic
	in this ball.  By Baire's theorem, $f$ has 
	at least one point of continuity 
	and hence is holomorphic in this ball.  The derivatives of $f$ at $z$ are given by
	$$
	\mdiff{m}{f}(z)(w-z)=\sum_{|\alpha|=m}
	c_\alpha (w-z)^\alpha \,.
	$$
	
	Choose $0< \sigma< \rho$
	so that the holomorphic function
	$\sum_\alpha |c_\alpha| (w-z)^\alpha$
	is bounded by $C>0$ in the ball
	$\|w-z\|\le \sigma$. 
	For every non-zero $\xi \in E$,
	$$
	\sum_{|\alpha|=m } |c_\alpha||\xi|^\alpha
	= \Bigl(\frac{\|\xi\|}{\sigma}\Bigr)^m \sum_{|\alpha|=m }
	|c_\alpha| \bigl(\sigma
	\, |\xi|/\|\xi\|\bigr)^\alpha \le C\Bigl(\frac{\|\xi\|}{\sigma}\Bigr)^m\,.
	$$
	Thus $\mdiff{m}{f}(z)$
	is a regular
	$m$-homogeneous polynomial for 
	every $m\in \N$ and the inequality
	above shows that 
	$$
	\Bigl\| \mdiff{m}{f}(z)\Bigl\|_r
	=
	\sup\bigl\{\sum_{|\alpha|=m}|c_\alpha|
	|\xi|^\alpha : \|\xi\| \le 1 \bigr\}
	\le C\, \Bigl(\frac{1}{\sigma}\Bigr)^m
	$$
	for every $m\in \N$.  Therefore $f$ is regularly
	holomorphic on $U$.
	
\end{proof}

Suppose that $f$ is a
holomorphic function defined on an open subset $U$ of a complex Banach 
space with an unconditional Schauder basis $E$ and that $z$ belongs to $U$.
Then it follows from the proof of the above theorem that $\displaystyle{
	\frac{|r|
		(f,z)}{r(f,z)}}$ measures the ratio between the radius of convergence of
the mononial expansion of $f$ about $z$ and the radius of convergence of the
Taylor series of $f$ about $z$.

Let $E, F$ be complex Banach lattices, with $F$ Dedekind complete and let $U$ be an open subset
of $E$.  A function $f\colon U\to F$ is
\textit{regularly holomorphic of bounded
type} if it is regularly holomorphic
on $U$ and, for every $z\in U$,
\begin{equation}\label{e:bounded}
	\radreg{f}{z} =
	\left(
	\limsup_m \,\Bigl\| \mdiff{m}{f}(z)\Bigl\|_r^{\frac{1}{m}}
	\right)^{-1} \ge d(z, \partial U)\,.
\end{equation}
Thus, the Taylor series of $f$ at 
every point $z$ in $U$ is required to be
uniformly regularly convergent on
every ball with centre $z$ that is contained
in $U$.  We denote the space 
of regularly holomorphic functions
of bounded type by $\rbhol{U}{F}$,
or by $\rbhols{U}$ when $F=\C$.

\begin{proposition}
	Let $E, F$ be complex Banach lattices
	with $F$ Dedekind complete, let $U$ 
	be a solid open subset of $E$ and
	let $f\colon U \to F$ be 
	a regularly holomorphic mapping 
	of bounded type.
	Then, for every $z\in U$, 
	the Taylor series of $f$ at
	$z$ is uniformly regularly 
	convergent on every polydisc
	with centre $z$ contained in $U$.
\end{proposition}

\begin{proof}
This follows from the fact that
$$
\Bigl| \mdiff{m}{f}(z)\Bigl|(w-z)
\le
\Bigl| \mdiff{m}{f}(z)\Bigl|(a)
$$
whenever $a\in E_+$ and $w\in D(z,a)$.
\end{proof}

\bigskip
We recall the definition of a holomorphy type, 
introduced by Nachbin \cite{Nachbin}.  

\begin{definition}
Let $\EuScript{B}$ be a class of ordered pairs of Banach spaces.
A holomorphy type on $\EuScript{B}$ is an
assignment, to each $(E,F)$ in $\EuScript{B}$, of
a sequence of Banach spaces
$(\thetapoly{m}{E}{F},\|\cdot\|_\theta )_m$	
with the following properties:
\begin{enumerate}
	\item[(a)] Each $\thetapoly{m}{E}{F}$ is a vector
	subspace of $\hpoly{m}{E}{F}$.
	\item[(b)] $\thetapoly{0}{E}{F}$ coincides with 
	$\hpoly{0}{E}{F}=F$ as a Banach space.
	\item[(c)] There is a real number $\sigma\ge 1$ such that,
	if $k,m$ are natural numbers with $k\le m$, $a\in E$
	and $P\in\thetapoly{m}{E}{F}$, then
	$\widehat{d}^k P \in \thetapoly{m-k}{E}{\thetapoly{k}{E}{F}}$ and
	$$
	\Bigl\| \frac{1}{k!}\widehat{d}^k P (a) \Bigr\|_\theta \le
	\sigma^m \|P\|_\theta \|a\|^{m-k}\,.
	$$
	
\end{enumerate}
\end{definition}

\begin{proposition}[cf. \cite{Matos}, Prop.~3.8 ]
	The sequence $\displaystyle \bigl(\rhpolys{m}{E,F}, \|\cdot\|_r\bigr)_m$
	is a holomorphy type on the class of pairs $(E,F)$ of
	complex Banach lattices with $F$ Dedekind complete.
\end{proposition}

\begin{proof}
	Properties (a) and (b) are obvious.  To show that (c) is satisfied,
	let $P = \widehat{A}\in \rhpoly{m}{E}{F}$. Then, using
	Proposition \ref{DiffHom} and 
	the polarization inequality,
	\begin{align*}
	\Bigl\| \mdiff{k}{P}(a) \Bigl\|_r  &=
		\Bigl\|  \;\Bigl|\mdiff{k}{P}(a)\Bigr| \;\Bigl\| \leq
		\Bigl\| \mdiff{k}{|P|}(a) \Bigl\|\\
	&= \binom{m}{k} \Bigl\| \;|A|a^{m-k}\Bigr| \;\Bigr\|\le
	\binom{m}{k}\frac{m^m}{m!}\| \;|P|\;\| \|a\|^{m-k}\\
	&\le \sigma^m\, \|P_m\|_r\, \|a\|^{m-k}
	\end{align*}
	where $\sigma =2e$.
\end{proof}

\bigskip

In some circumstances, it is possible to define a complex
lattice structure on the space of regularly holomorphic
functions.  We shall use the axiomatization of complex
vector lattices given by Mittelmeyer--Wolff \cite{MW}.
They define an  \emph{Archimedean modulus} on a complex
vector space $G$ to be a 
function $m\colon G\to G$ satisfying
\begin{itemize}
	\item[(0)] $m(m(x)) = m(x)$ for every $x\in G$.
	\item[(i)] $m(\alpha x)= |\alpha| m(x)$
	for every $\alpha\in \C$, $x\in G$.
	\item[(ii)]
	$m\bigl( m( m(x)+m(y)) -m(x+y)\bigr) = m(x)+m(y)-m(x+y)$.
	\item[(iii)] $G$ is the linear span of $m(G)$.
	\item[(iv)] 
	$m(m(y)-km(x)) = m(y) -km(x)$ for every $k\in \N$ implies $x=0$.
\end{itemize}

A \emph{complex vector lattice} is a complex vector space $G$
equipped with a modulus.  
It is shown in \cite{MW}  that the real linear
span of the subset $m(G)$ of $G$ is a real vector lattice
with $m(G)$ as the positive cone.

\begin{proposition}
	Let $E$ be a complex Banach lattice
	and let $U$ be an open ball in $E$
	with centre $z_0$. Then the space $\rbhols{U}$ of
	regularly holomorphic functions of 
	bounded type on $E$ is a complex Archimedean
	Riesz space with the modulus 
	given by
	$$
	m(f) =
	|f| := \sum_{m=0}^{\infty} 
	\Bigl| \frac{1}{m!} \widehat{d}^m f(z_0)  \Bigr|\,. 
	$$
\end{proposition}

\begin{proof}
	The proof consists of straightforward calculations.  For example, to show
	that property (ii) in the definition is satisfied,
	let $f=\sum_m P_m$ and $g=\sum_m Q_m$ be 
	the Taylor expansions at $z_0$. Then,
	since $m$ coincides with the modulus 
	on $\rhpolys{m}{E}$,
	\begin{multline*}
	m\bigl( m( m(f)+m(g)) -m(f+g)\bigr)\\
	= 
	\biggl| \Bigl| \sum_m |P_m| + |Q_m|\Bigr| 
	-\sum_m |P_m + Q_m|
	\biggr|
	=
	\biggl|  \sum_m |P_m| + |Q_m| 
	-|P_m + Q_m|
	\biggr|\\
	= \sum_m |P_m| + |Q_m| 
	-|P_m + Q_m| = 
	m(f)+m(g)- m(f+g)\,.
	\end{multline*}
\end{proof}

Is it possible to define a complex lattice 
structure on $\rbhols{U}$ for domains
other than open balls?
It would be reasonable to expect
that, if $B(a,r)$ is an open 
ball contained in  $U$, then the restriction
mapping from $\rbhols{U}$ to $\rbhols{B(a,r)}$
would be a lattice homomorphism. If
$\rbhols{U}$ carried a complex
lattice structure satisfying this condition,
then it would follow that 
$$
|f|=\sum_{m=0}^{\infty} 
\Bigl| \frac{1}{m!} \widehat{d}^m f(a)  \Bigr| 
$$
in the ball $B(a,r)$.
\color{black}
However, this attempt to define $|f|$ locally
fails, due to a lack of coherence.

To see this, take 
$E=\C$, 
and consider a holomorphic function $f$ on $B(a,r)$. Then for each point $z_0$ in $B(a,r)$ we can expand $f$ as a
Taylor series
$$
f(z)=\sum_{m=0}^\infty a_m(z-z_0)^m
$$
for $|z|<r-|z_0|$.  Note that the coefficients $(a_m)_m$ depend on the point
$z_0$ and are given by $\displaystyle a_m=\frac{f^{(m)}(z_0)}{m!}$.

For each point $z_0$ we define the holomorphic function $|f|_{z_0}$ by
$$
|f|_{z_0}(z)=\sum_{n=0}^\infty |a_n|(z-z_0)^n.
$$
Note that this series converges on $B(z_0,r-|z_0|)$.

However, 
these expansions lack
 coherence. To see this take the
function $f\colon B(0,1)\to\mathbb{C}$ given by
$$
f(z)=\frac{1}{1-z}.
$$
We first expand $f$ about the origin to get the Taylor series expansion
$$
f(z)=\sum_{m=0}^\infty z^m.
$$
The Taylor series expansion of $f$ about $\frac{i}{2}$ is
$$
f(z)=\sum_{m=0}^\infty \left(\frac{1}{(1-\frac{i}{2})^{m+1}}\right)\left(z
-\frac{i}
{2}\right)^m=\sum_{m=0}^\infty \left(\frac{2}{5}(2+i)\right)^{m+1}\left(z
-\frac{i}{2}\right)^m.
$$ 
Then
$$
|f|_0(z)=f(z)=\sum_{m=0}^\infty z^m
$$
and
$$
|f|_{\frac{i}{2}}(z)=\sum_{m=0}^\infty \left(\frac{2}{\sqrt{5}}\right)^{m+1}\left(z-
\frac{i}{2}\right)^m.
$$
We now observe that $|f|_0(\frac{i}{2})=\frac{2}{5}(2+i)$ while $|f|_{\frac{i}
	{2}}(\frac{i}{2})=\frac{2}{\sqrt{5}}$.

So it appears that, in general, it is not
possible to define a complex lattice
structure on $\rbhols{U}$.

\section{The Bohr radius on Banach lattices}

In this section we show how the Bohr radius
for complex sequence spaces can be generalized
to the wider setting of complex Banach lattices.  In Theorem \ref{Bohr} we prove
that the growth of the Bohr radii controls
the radius of regular convergence and we show
that on finite dimensional sequence spaces, the radii of convergence and 
regular convergence are equal.

In 1914 H. Bohr \cite{Bohr} showed that if $\sum_{k=0}^\infty c_kz^k$ is a
power series on the unit disc centred at $0$, $D(0,1)$, with $\left|\sum_{k=0}
^\infty c_kz^k\right|\le 1$ for all $z$ with $|z|<1$ then $\sum_{k=0}^\infty
|c_kz^k|\le 1$ for all $z$ with $|z|\le \frac{1}{3}$. Moreover, $\frac{1}{3}$
is the optimal radius for which this inequality holds.

More generally, given a Reinhardt domain $R$ in $\mathbb{C}^n$ the Bohr radius
of $R$, $K(R)$, is defined as the supremum over all $r\ge 0$ such that if
$\sum_\alpha c_\alpha z^\alpha$ is a power series on $R$ with $\left|
\sum_\alpha c_\alpha z^\alpha\right|\le 1$ for all $z$ in $R$ then
$\sum_\alpha|c_\alpha z^\alpha|\le 1$ for all $z$ in $rR$. It follows from
results of Aizenberg \cite{Aiz}, Boas \cite{Boas}, Boas and
Kavinson \cite{BK} and Dineen and Timoney \cite{DT}, that there is a constant
$c$, independent of $n\in\mathbb{N}$ and $1\le p\le\infty$, such that
$$
\frac{1}{c}\left(\frac{1}{n}\right)^{1-\frac{1}{\min(p,2)}}\le K(B_{\ell_p^n})
\le c\left(\frac{\log n}{n}\right)^{1-\frac{1}{\min(p,2)}}.
$$

In 2003, Defant, Garc{\'\i}a and Maestre \cite{DGM2003} refined the concept of
Bohr radius and introduced the concept of homogeneous Bohr radius. Given a
finite dimensional Banach space $X=(\mathbb{C}^n,\|\cdot\|)$ with canonical
basis $(e_k)_{k=1}^n$ they define $K_m(B_X)$ as the supremum over all $r$ in
$[0,1]$ such that if $\sum_{|\alpha|=m}c_\alpha z^\alpha$ is an $m$-homogeneous
polynomial on $X$ with $\left|\sum_{|\alpha|=m}c_\alpha z^\alpha\right|\le 1$
for all $z$ in $B_X$ then $\sum_{|\alpha|=m}|c_\alpha z^\alpha|\le 1$ for all
$z$ in $rB_X$.

Let $E$ be a Banach space with an unconditional basis, $(x_n)_n$. Then the unconditional basis constant of $(x_n)_n$, $\chi((x_n)_n)$, 
is defined by
$$
\chi((x_n)_n)=\inf\left\{C:\left\|\sum_{k=1}^\infty \epsilon_k\mu_k x_k\right
\|\le C\left\|\sum_{k=1}^\infty \mu_kx_k\right\|: \mu_j\in \mathbb{C},
|\epsilon_j|=1, j\in\mathbb{N}\right\}.
$$

In \cite[Lemma~2.1]{DGM2003} it is shown that if $X=(\mathbb{C}^n,\|
\cdot\|)$ then       
$$
K_m(B_X)=\frac{1}{\root m \of {\chi_{mon}(\hpolys{m}{X})}}
$$
where $\chi_{mon}(\hpolys{m}{X})$ denotes the unconditional basis constant
of the monomials in $\hpolys{m}{X}$.

Given any $m$-homogeneous polynomial $P=\sum_{|\alpha|=m}c_\alpha z^\alpha$,
when we regard $\hpolys{m}{X}$ as a Banach lattice, the absolute value of
$P$ is given by $|P|=\sum_{|\alpha|=m}|c_\alpha|z^\alpha$. Motivated by this
observation we now introduce the $m$-th Bohr radius of a general Banach
lattice.

\begin{definition}
	Let $E$ be a complex Banach lattice and $m$ be a positive integer. We define the
	$m$-th Bohr radius of $E$, $K_m(B_E)$ by
	\begin{align*}
	K_m(B_E)&:=\sup\{\rho:\sup_{z\in \rho B_E}||P|(z)|\le \|P\| \hbox{ for all }
	P\in \rhpolys{m}{E}  \}\\
	&=
	\sup\{\rho: 
	\rho^m \|P\|_r \le \|P\|
	\text{ for all $P\in \rhpolys{m}{E}$}
	\}.
	\end{align*}
\end{definition}

Note that $0\le K_m(B_E)\le 1$ and that $K_m(B_E)$ may be $0$.

Indeed, we observe that $K_m(B_E)>0$ if and only if the regular and supremum norms are
equivalent when restricted to $\rhpolys{m}{E}$. In this case, 
${K_m(
	B_E)}^{-m}$ 
\color{black}
is the norm of the identity mapping from $(\rhpolys{m}{E},
\|\cdot\|)$ onto $(\rhpolys{m}{E},\|\cdot\|_r)$.

Let $E$ be a complex Banach lattice, $U$ be an open subset of $E$ and $f\colon U\to
\mathbb{C}$ be a homomorphic function such that $f(a+z)=\sum_{m=0}^\infty
P_m(z)$ is the Taylor series of $f$ about $a$ with each $P_m$ a regular
$m$-homogeneous polynomial.  Our next theorem provides  a lower
bound for
the radius of regular convergence,
$|r|(f,a)$, in terms of the
homogeneous Bohr radii.

\begin{theorem}\label{Bohr}
	Let $E$ be a complex Banach lattice
	and $U$ an open subset of $E$.  Let $f\colon U\to \mathbb{C}$ be a regular
	holomorphic function and $a\in U$. Then,
	$$
	\liminf_{m\to\infty}(K_m(B_E))r(f,a)\le |r|(f,a)\le r(f,a).
	$$
	Moreover, for each $a\in U$, both the upper and lower bounds are sharp.
	
\end{theorem}

\begin{proof} If $\liminf_{m\to\infty}K_m(B_E)=0$ then there is nothing to
	prove. Let us now suppose that there is $m_0\in \mathbb{N}$ so
	that  $K_m(B_E)>0$ for each $m\geq m_0$. Write the Taylor series of $f$
	about $a$ as
	$$
	f(a+z)=\sum_{m=0}^\infty P_m(z).
	$$
	Then, by the definition of the Bohr radius, we have that
	$$
	(K_m(B_E))^m\|P_m\|_r\le \|P_m\|
	$$
	for each $m$ which we will rewrite as
	$$
	\|P_m\|_r\le \frac{1}{(K_m(B_E))^m}\|P_m\|
	$$
	for all $m\ge m_0$. Taking $m$-th roots we get that
	$$
	\|P_m\|_r^{1/m}\le \frac{1}{(K_m(B_E))}\|P_m\|^{1/m}.
	$$
	We now let $m$ tend to infinity to get
	$$
	\limsup_{m\to\infty}\|P_m\|_r^{1/m}\le \limsup_{m\to\infty}\frac{1}{(K_m(B_E))}
	\limsup_{m\to\infty}\|P_m\|^{1/m}.
	$$
	Noting that 
	$\displaystyle\frac{1}{\limsup_{m\to \infty}a_m}=\liminf_{m\to\infty}
	a_m$
	\color{black}
	 we get that
	$$
	\limsup_{m\to\infty}\|P_m\|_r^{1/m}\le \frac{1}{\liminf_{m\to\infty}(K_m(B_E))}
	\limsup_{m\to\infty}\|P_m\|^{1/m}
	$$
	and inverting we obtain
	$$
	\liminf_{m\to\infty}(K_m(B_E))r(f,a)\le |r|(f,a).
	$$
	Therefore, we have
	$$
	\liminf_{m\to\infty}(K_m(B_E))r(f,a)\le |r|(f,a)\le r(f,a).
	$$

	Let us now see that both sides of this inequality are sharp. If we
	consider a
	holomorphic function $f=\sum_{m=0}^\infty P_m$ with $P_m\ge 0$ for each $m\in \mathbb{N}$ then we have that $|r|(f,0)=r(f,0)$.
	
	To show that the lower bound is optimal we consider two cases.
	We first assume that $K_m(B_E)>0$ all
	but finitely many $m$. Then, discarding finitely many terms, we can assume
	that
	$K_m(B_E)>0$ for all $m$. For each $m\in \mathbb{N}$ we choose $P_m\not=0$
	so that
	\begin{equation}\label{eq:C}
	\frac{1}{2}\|P_m\|\le K_m(B_E)^m\|P_m\|_r\le \|P_m\|. 
	\end{equation}
	Dividing $P_m$ by $\|P_m\|$ we may assume that $\|P_m\|=1$ for each $m\in \mathbb{N}$. If we let $f=\sum_{m=1}^\infty P_m$ then $r(f,0)=1$.
	
	From (\ref{eq:C})  we obtain 
	$$
	\left(\|P_m\|^{1/m}\right)^{-1}\le K_m(B_E)^{-1}(\|P_m
	\|_r^{1/m})^{-1}\le 2^{1/m}(\|P_m\|^{1/m})^{-1} \,,
	$$
	or
	$$
	K_m(B_E) \left(\|P_m\|^{1/m}\right)^{-1}\le (\|P_m\|_r^{1/m})^{-1}\le 2^{1/m}
	K_m(B_E)(\|P_m\|^{1/m})^{-1}.
	$$
	As $\|P_m\|=1$ for all $m$ we get
	$$
	K_m(B_E)\le (\|P_m\|_r^{1/m})^{-1}\le 2^{1/m} K_m(B_E).
	$$
	Letting $m$ tend to infinity we get that
	$$\left(\limsup_{m\to\infty}\|P_m\|_r^{1/m}\right)^{-1}=
	\liminf_{m\to\infty}K_m(B_E)
	$$
	and therefore we see that the lower bound is attained.
	
	Let us now suppose that we have infinitely many $m$ with $K_m(B_E)=0$.
	Then we can choose a
	subsequence $(m_k)_k$ so that $K_{m_k}(B_E)=0$ for all $k$. Fix $j\in
	\mathbb{N}$ and for each $k$ in $\mathbb{N}$ choose $P_k\in
	\rhpolys{m_k}{E} $ so that
	$$
	\sup_{z\in \frac{1}{j}B_E}||P_k|(z)|> \|P_k\|
	$$
	or that 
	$$
	\|P_k\|_r>  j^{m_k}\|P_k\|.
	$$
	Then repeating the above argument with $K_m(B_E)$ replaced with $\frac{1}{j}$
	we get
	$$
	\limsup_{k\to\infty}\|P_k\|_r^{1/m_k}\ge\liminf_{m_k\to\infty}{j}
	$$
	and therefore
	$$
	|r|(f,0)\le \frac{1}{j}.
	$$
	As this holds for all $j\in \mathbb{N}$ we see that we can find regular holomorphic
	functions with radius of convergence $1$ but radius of absolute convergence
	arbitrarily small.
\end{proof}

As we have mentioned above, it follows from \cite[Lemma~2.1]{DGM2003} that if 
$X$ is $(\mathbb{C}^n,\|
\cdot\|)$ then 
$K_m(B_X)= \frac{1}{\root m \of {\chi_{\rm mon}(\hpolys{m}{X})}}$.
In this case Theorem~\ref{Bohr} gives us that
$$
\frac{1}{\limsup_{m\to\infty}\root m \of {\chi_{\rm mon}(\hpolys{m}{X})}}r(f,a)\le
|r|(f,a)\le r(f,a).
$$

If $X$ is the finite dimensional Banach lattice $(\mathbb{C}^n,\|\cdot\|)$,
with any norm for which
the unit vector basis
is a $1$-unconditional
Schauder basis, then it
follows from \cite[Lemma~11..2.2]{DGMMM} that $\limsup_{m\to \infty}\sqrt[m]{\chi_{\rm
	mon}(\hpolys{m}{X})}=1$. Hence, we have the following result.

\begin{theorem}\label{finiteradius}
	Let $f$ be a holomorphic function on $(\mathbb{C}^n,\|\cdot\|)$,
	with any norm for which
	the unit vector basis
	is a $1$-unconditional
	Schauder basis. Then $r(f,a)=|r|(f,a)$ for every $a$.
\end{theorem}

Note that we can restate the above theorem as follows:
		Let $f$ be a holomorphic function on $(\mathbb{C}^n,\|\cdot\|)$,
		with any norm for which
		the unit vector basis
		is a $1$-unconditional
		Schauder basis. Then for each $a$ in $\mathbb{C}^n$ the radius of
		convergence of the monomial expansion of $f$ about $a$ is equal to
		the radius of convergence of the Taylor series of $f$ about $a$.

\bigskip

Theorem~\ref{finiteradius} is not true for real analytic functions on subsets
of $\mathbb{R}^n$. Indeed, Hayman, \cite{Hay}, shows that if $f(x)=
\sum_{k=0}^\infty P_k(x)$, where each $P_k$ is a harmonic $k$-homogeneous polynomial, converges on the 
polydisc $\{(x_i)_{i=1}^n:|x_i|<r\}$ in $\ell_\infty^n$ then $\sum_{k=0}^\infty
|P_k|$ converges on the polydisc $\{(x_i)_{i=1}^n:|x_i|<r/\sqrt{2}\}$. 
Moreover, an
example is provided in \cite{Hay} to show that the factor of $r/\sqrt{2}$ is
sharp.

\section{Regular Holomorphic Functions on $\ell_p$}

In this section we will look at regular holomorphic functions on $\ell_p$ for
$1<p<\infty$. We will show that it possible to construct holomorphic functions
with radius of convergence $1$ yet having arbitrarily small radius of regular
convergence.

In our constructions we need to consider 
the Bohr radius, $K_m(B_{\ell_p})$, of $\hpolys{m}{\ell_p}$.
\color{black}
We begin with the observation that for any positive integer $k$ we have
$K_m(B_{\ell_p})\le K_m(B_{\ell_p^k})$. In addition, by \cite[Lemma~2.1]{DGM2003}, for each $m$ and $k$ we have
$$
K_m(B_{\ell_p^k})=\frac{1}{\root m \of {\chi_{\rm mon}(\hpolys{m}{\ell_p^k})}}.
  $$
Let $\pi_k$ denote the canonical
projection from $\ell_p$ onto
$\ell_p^k$.
Our plan is to define a holomorphic
function $f$ on $\ell_p$ of the form $f(z)=\sum_{m=0}^\infty P_m\circ\pi_{n_m}
(z)$ where for each $m$, $P_m$ is an $m$-homogeneous polynomial on
$\ell_p^{n_m}$ with $n_m\to\infty$ as $m\to\infty$.

We need some notation. Let $(A_{mn})_{m,n}$ and $(B_{mn})_{m,n}$ be two doubly
indexed sequences of positive real numbers. Following \cite{DF}, we will write $A_{mn}\sim B_{mn}$
if there is $C\ge 1$ so that for all $m$ and $n$ we have
$$
(1/C^m) A_{mn}\le B_{mn}\le C^m A_{mn}.
$$

It follows from
\cite[Page 133]{DF} that for $1\le p\le \infty$ we have
$$
\chi_{\rm mon}(\hpolys{m}{\ell_p^n})\sim \left(1+\frac{n}{m}\right)^{(m-1)
  \left(1-\frac{1}{\min(p,2)}\right)}.
  $$

 Let $\alpha_m>0$.
  Then for each fixed positive integer $m$, let us choose 
  $n_m$ 
  so that
  $$
  1+\frac{n_m}{m}> \left((1+\alpha_m)^{\frac{m}{m-1}}\left(C^{\frac{m}{m-1}}
    \right)\right)^{\frac{1}{1-\frac{1}{\min(p,2)}}}.
  $$
  Then we have
  \begin{align*}
    \root m\of{\chi_{\rm mon}(\hpolys{m}{\ell_p^{n_m}})}&\ge
    \frac{1}{C}\left(\left(1+\frac{n_m}{m}\right)^{1-\frac{1}{\min(p,2)}}\right)
                                                             ^{\frac{m-1}{m}}\\
                                                           &> 1+\alpha_m.
  \end{align*}
  Therefore,
  for
  each positive integer $m$ we have that $K_m(B_{\ell_p^{n_m}})<
  \frac{1}{1+\alpha_m}$.

  Thus,
  for each $m$ we can choose an $m$-homogeneous polynomial $P_m$ in
  $\hpolys{m}{\ell_p^{n_m}}$ so that $\|P_m\|=1$ and $\|P_m\|_r^{\frac{1}{m}}>
  {1+\alpha_m}$. If we define
  $f$ on the unit ball of $\ell_p$ by $f(z)=\sum_{n=0}^\infty P_m(\pi_{n_m}(z))$
  then
  $r(f,0)=1$ while $|r|(f,0)\le \left(\limsup_{m\to\infty}\frac{1}{1+\alpha_m}
    \right)^{-1}$.

  Let us show that if $p>1$ then for each $m$ in $\mathbb{N}$ and each $1\le
  \eta<\infty$ we can find $n_m\in \mathbb{N}$ and an $m$-homogeneous
  polynomial $P_\eta$ in $\hpolys{m}{\ell_p^{n_m}}$
  with $\|P_\eta\|=1$ and $\|P_\eta\|_r^{\frac{1}{m}}=\eta$. 

  We begin by choosing $\alpha_m$ so that $1+\alpha_m>\eta$.
  From the above construction we know that we can find $n_m\in \mathbb{N}$ and
  $P_m$ in $\hpolys{m}{\ell_p^{n_m}}$ with $\|P_m\|=1$ and $\|P_m\|_r^{1/m}> 1+\alpha_m$. We now choose a
positive $m$-homogeneous polynomial $Q_m$ on $\hpolys{m}{\ell_p^{n_m}}$ with
$\|Q_m\|=1$. As $S_{\hpolys{m}{\ell_p^{n_m}}}$ is path connected, we can find
a path $\gamma\colon [0,1]\to S_{\hpolys{m}{\ell_p^{n_m}}}$ with $\gamma(0)=Q_m$
and $\gamma(1)=P_m$.

Let us now consider the function $\lambda\colon [0,1]\to\mathbb{R}$ given by
$\lambda(t)=\|\gamma(t)\|_r^{\frac{1}{m}}$. Since the supremum and regular norms
are equivalent, we know that $\lambda$ is continuous. As $\lambda(0)=1$ and
$\lambda
(1)\ge 1+\alpha_m>\eta>1$ the Intermediate Value Theorem tells us that we can
find $t_o\in (0,1)$ so that $\lambda(t_o)=\eta$. If we set 
$\gamma(t_0)=
P_{m,\eta}$
then we have an $m$-homogeneous polynomial $P_{m,\eta}$ on $\ell_p^{n_m}$ with
$\|P_{m,\eta}\|=1$ and $\|P_{m,\eta}\|_r^{\frac{1}{m}}=\eta$. We now define
$f$ on $B_{\ell_p}$ by $f(z)=\sum_{m=0}^\infty P_{m,\eta}(\pi_{n_m}(z))$ to obtain
the following result.

\begin{proposition}\label{ellp}
Let $p>1$. Then for each 
$\tau\in (0,1)$ 
there is a holomorphic function $f$ on $B_{
\ell_p}$ with $r(f,0)=1$ and
$|r|(f,0)=\tau$.
\end{proposition}

This proposition tells us that on $\ell_p$, $1<p<\infty$, for
every $\tau$ in $(0,1)$,
it is possible to find a holomorphic function that has a Taylor series about
$0$ with radius of convergence equal to $1$, yet whose monomial expansion
about $0$ has radius of convergence  equal to $\tau$.

For $\ell_1$, Matos \cite[Prop.~3.7]{Matos}
showed that every bounded $m$-homogeneous 
polynomial is regular  and that
the regular and supremum norms 
on $\npoly{m}{\ell_1}$ are equivalent.
More precisely,
$$
\|P\|\le \|P\|_r \le e^m \|P\|
$$
for every $P\in \npoly{m}{\ell_1}$.
Thus, on $\ell_1$, we have
$e^{-1}r(f,0) \le |r|(f,0) \le r(f,0)$ for every
holomorphic function $f$ on $\ell_1$ or its unit ball.
We do not know if this inequality is 
sharp.
  
\section{Orthogonally Additive Holomorphic Functions}

In this 
section
we 
study 
orthogonally additive holomorphic functions on
Banach lattices and show that for this class of functions the radii of
convergence and regular convergence coincide.

We begin by looking at what is known for finite dimensional Banach lattices.
Let us suppose that $X$ is 
 $\mathbb{C}^n$ endowed with a norm under
which $(e_j)_{j=1}^n$ is a Schauder basis with unconditional basis constant $1$.
Given 
a subset $J$
of $\{1,\ldots,n\}^m$ for some natural number $m$, Bayart,
Defant and Schl\"uters, \cite{BDS},  use $\hpolys{J}{X}$ to denote the closed
subspace of all holomorphic function $f$ in $\mathcal{H}^\infty(B_X)$ for which
$c_\alpha(f)=0$ if $\alpha\not\in J$. 
In other words, ${\cal P}(^JX)$
is the span of $\{z^j;j\in J\}$.

In the case where we take 
$J=\{(k,\ldots,k) : k\in \{1,\dots,n\}\}$ 
the space
$\hpolys{J}{X}$ is the space of $m$-homogeneous orthogonally additive
polynomials on $X$. Using \cite[Theorem~1.3]{AG} Bayart, Defant and Schl\"uters,
deduce on \cite[Page 113]{BDS} that $\chi_{\rm mon}(\hpolys{J}{X})=1$. In
particular, this means that $\|P\|=\|P\|_r$ for every $m$-homogeneous
orthogonally
additive polynomial on $X$. We now extend this result to orthogonally additive
polynomials on any complex Banach lattice.

Let $E$ be a complex Banach lattice. Then $z,w\in E$ are said to be \emph{disjoint}, denoted by $z\perp x$, if $|z|\wedge |w| = 0$ in $E_\R$.

\begin{lemma}
	Let $E$ be a complex Banach
	lattice and let $z=x+iy, w=u+iv$
	be elements of $ E$.  Then $z\perp w$  if and only if $x,y\perp u,v$.
\end{lemma}
\begin{proof}
	We have
	\begin{multline}\label{1}
	|z|\wedge |w| = \\
	\sup\{|x|\cos\theta+|y|\sin\theta: \theta\in[0, 2\pi]\} \wedge \sup\{|u|\cos\phi+|v|\sin\phi:  \phi\in[0, 2\pi]\}\\
	= \sup_\theta \sup_\phi \;
	\bigl(|x|\cos\theta +|y|\sin\theta\bigr)\wedge
	\bigl(|u|\cos\phi +|v|\sin\phi \bigr)\,,
	\end{multline}
	using the infinite distributive property
	of the lattice operations 
	\cite[Theorem 1.8]{AB}.

	Suppose that $z\perp w$.  Then, by (\ref{1}),
	$$
	\bigl(|x|\cos\theta +|y|\sin\theta\bigr)\wedge
	\bigl(|u|\cos\phi +|v|\sin\phi \bigr) \le 0\,
	$$
	for all $\theta, \phi$.  
	Taking $\theta,\phi = 0,\pi/2$, we get
	$$
	|x|\wedge |u| = |x|\wedge |v| = |y|\wedge |u| 
	=|y|\wedge |v| = 0\,.
	$$
	
	The converse follows easily from (\ref{1}).
\end{proof}

	 We will now make use of the following property of
	 real orthogonally additive polynomials.  Let $Q$
	 be an  $m$-homogeneous polynomial on $E$, generated by the symmetric 
	 $m$-linear form $B$.  Then $B$ is
	 said to be \emph{orthosymmetric} if
	 $B(x_1,\dots,x_m)=0$ whenever any two of $x_1,\dots,x_m$
	 are disjoint.  
	 This property of $B$ is equivalent
	 to $Q$ being orthogonally additive
	 \cite{BB}.

	 \begin{proposition}\label{prop:oad}
	 	Let $E$ be a complex Banach lattice and $P$ be a regular $m$-homogeneous polynomial on $E$. Then the following are equivalent:
	 	\begin{itemize}
	 		\item[(a)] $P$ is orthogonally additive on $E$.
	 		\item[(b)] $P$ is orthogonally additive on $E_\R$.
	 		\item[(c)] The real and imaginary parts of $P$ are orthogonally additive on $E$.
	 	\end{itemize}
	 \end{proposition}

	 \begin{proof}
	 	The implications (a) implies (b) and (c) implies (a) are trivial.

	 	To show that (b) implies (c), suppose that $P=P_0+iP_1$ is orthogonally additive on $E_\R$.
	 	Then for $x,y \in E_\R$ with $x\perp y$, we have $P(x+y)=P(x) +P(y)$. Taking real and imaginary parts of $P$ and using (\ref{eq:RealIm}), we get
	 	$$
	 	P_0(x+y)=P_0(x) +P_0(y)\quad\mbox{and}\quad P_1(x+y)=P_1(x) +P_1(y).
	 	$$
	 	That is, $P_0$ and $P_1$ are orthogonally additive on $E_\R$.
	 	 
	 	To complete the proof, we must show that complexifications of $P_0$ and $P_1$ are orthogonally additive
	 	on $E$.
	 	Let $z=x+iy$, $w=u+iv$ be disjoint 
	 	elements of $E$. We have
	 	$$
	 		P_0(z+w) = \sum_{k=0}^m i^k\binom{m}{k}
	 		A_0(x+u)^{m-k}(y+v)^k\,,
	 	$$
 		where $A_0$ is the symmetric $m$-linear form
 		that generates $P_0$.  Expanding, we get
 		$$
 		P_0(z+w) = \sum_{k=0}^m i^k \binom{m}{k}
 		\sum_{r=0}^{m-k}\sum_{s=0}^k \binom{m-k}{r}
 		\binom{k}{s}A_0(x^{m-k-r}u^ry^{k-s}v^s) \,.
 		$$
 		Since $P_0$ is orthogonally additive on $E_\R$, the $m$-linear
 		form $A_0$ is orthosymmetric on $E_\R^m$.  Thus,
 		we have that 
 		$A_0(x^{m-k-r}u^ry^{k-s}v^s)$ vanishes if either
 		of $x,y$ appears along with either of $u,v$ in the
 		argument.  Therefore, for each $k$, there are only 
 		two non-zero terms, $A_0(x^{m-k}y^k)$
 		and $A_0(u^{m-k}v^k)$ and hence
 		$$
 		P_0(z+w) = P_0(z)+P_0(w)\,.
 		$$
A similar argument shows that
$P_1$ is orthogonally additive on $E$.    \end{proof}

 It follows from the above proposition 
 that the Banach lattice of orthogonally additive $m$-homogeneous polynomials on a complex Banach lattice $E$ is the complexification of the lattice of orthogonally additive $m$-homogeneous polynomials on $E_\R$.

A regular holomorphic function $f\colon E\to\mathbb{C}$ is said to be
orthogonally additive if $f(x+y)=f(x)+f(y)$ whenever $x$ and $y$ are disjoint.
Orthogonally additive holomorphic functions have previously been defined on
$C(K)$ spaces by Carando, Lassalle and Zalduendo, \cite{CLZ2} and on
$C^*$-algebras by Jaramillo, Prieto and Zalduendo, \cite{JPZ}  and Peralta and
Puglisi, \cite{PP}. The proof of \cite[Lemma~1.1]{CLZ2} for $C(K)$ spaces
trivially extends to arbitrary complex Banach lattices to show that if
$f\colon E\to \mathbb{C}$ is a regular holomorphic function which has a Taylor
series expansion
$f(z) =\sum_{k=0} P_ k(z)$ about $0$ then $f$ is orthogonally additive if and
only if $P_k$ orthogonally additive for each $k$ in $\mathbb{N}$. 

In \cite{BRSGeom} the authors proved that if $P$ is an $m$-homogeneous
orthogonally additive polynomial on a real Banach lattice then $\|P\|=\|P\|_r$
when $m$ is odd and $\|P\|\le \|P\|_r\le 2\|P\|$ when $m$ is even. Moreover,
the polynomial $P(x)=x_1^m-x_2^m$, with $m$ even, on $\ell_\infty^2$ shows that this bound is
sharp.

Let us now consider what happens for orthogonally additive polynomials on
complex Banach lattices.

\begin{proposition}\label{C(K)}
Let $P$	be an orthogonally additive
$m$-homogeneous polynomial
on the complex Banach lattice $C(K)$.
Then $\|P\|_r = \|P\|$.
\end{proposition}

\begin{proof}
	
Using \cite{ABext} we know that every orthogonally 
additive $m$-homogeneous polynomial $P$  on $C(K)$ has an extension,
$\tilde P$, to $C(K)''$, the bidual of $C(K)$. Moreover, it follows from
\cite{DG} that $\|\tilde P\|=\|P\|$ and from \cite[Corollary~2.1]{CLZ2006} that
$\tilde P$ is orthogonally additive. We observe that $\mathcal{B}(K)$, the space
of all bounded Borel measurable functions on $K$ is a closed subspace of
$C(K)''$.
This means that each $m$-homogeneous orthogonally additive polynomial $P$ on
$C(K)$ has an extension $P_B$ to $\mathcal{B}(K)$ as an orthogonally additive
polynomial with $\|P_B\|=\|P\|$.

Since $P$ is orthogonally additive on $C(K)$, there is a complex measure
$\mu$ on $K$ such that
\begin{equation}\label{eq:ortcomplex}
P(x)=\int_K x^m(t)\, d\mu(t)
\end{equation}
for all $x$ in $C(K)$. We observe that (\ref{eq:ortcomplex}) is also valid in
	the real case. From \cite[Proposition~3]{BRSGeom} it now follows that $|P|$
	is represented by the measure $|\mu|$. The identification of $\mu$ with $P$
	gives us a Banach lattice
	isomorphism  from the the space of regular Borel signed measures on $K$ to the
	space of orthogonally additive $m$-homogeneous polynomials on the real Banach
	lattice $C(K)$, see \cite[Theorem~3]{BRSGeom}. As the complexification of a real lattice isomorphism is a complex lattice isomorphism, we see that $|P|$ is represented by the measure $|\mu|$ in
	the complex case as well.

It follows from
\cite[Corollary~2.1]{CLZ2006}
that
$$
P_B(y)={\tilde P}(y)=\int_K y^m(t)\, d\mu(t)
$$
for all bounded Borel measurable functions $y$ on $K$.

By \cite[Theorem~6.12]{Rudin}, there exists a Borel measurable function
$\rho\colon K\to \mathbb{C}$ with $|\rho(t)|=1$ for all $t$ in $K$ such that
$d|\mu|=\rho d\mu$.

Let us choose a branch $\rho^{1/m}$ of the $m$-th root of $\rho$. For each $x$
in $C(K)$, we note that $x\rho^{1/m}$ is a bounded Borel function on $K$.
Therefore, given an orthogonally additive $m$-homogeneous polynomial $P$ on
$C(K)$ we have
\begin{align*}
  |P|(x)&=\int_Kx(t)^m\,d|\mu|(t)\\
      &=\int_Kx(t)^m\rho(t)\,d\mu(t)\\
      &=\int_K(x\rho^{1/m}(t))^m\,d\mu(t)\\
     &=P_B(x\rho^{1/m})     .
\end{align*}

Thus, we have
\begin{align*}
  \|P\|_r&=\sup_{\|x\|\le 1, x\in C(K)}||P|(x)|\\
         &=\sup_{\|x\|\le 1,x\in C(K)}|{P_B}(x\rho^{1/m})|\\
         &\le\sup_{\|y\|\le 1, y\in B(K)}|{P_B}(y)|\\
         &=\|P\|.
\end{align*}
As $\|P\|\le \|P\|_r$ 
in general, we have $\|P\|=\|P\|_r$.

\end{proof}

We can extend this result to all Banach lattices by
localising to principal ideals.
If $E$  is a real or complex  Banach lattice, then for $a\in E_+$,
the principal ideal $E_a$ is the ideal in $E$ consisting of all $x$ 
satisfying $|x|\le c\,a$ for some $c>0$. With the 
norm
$$
\|x\|_a = \inf\{c: |x|\le c\,a\}\,,
$$
it becomes a Banach lattice.  
As it
is an AM-space with unit, the Kakutani Representation
Theorem says that $E_a$  is Banach lattice
isometrically isomorphic to $C(K_a)$ for some 
compact Hausdorff space $K_a$. It is easy to see
that, in the complex case, $E_a$ is the Banach
lattice complexification of the real
Banach lattice $(E_\R)_a$.

The norm of a real or complex Banach lattice $E$ is determined by the norms
of its principal ideals.  For every $z\in E$, we have
$$
\|z\| = \inf\{\|a\|\,\|z
\|_a: a\in E, a>0  \}
$$
and this infimum is attained when
$a=|z|$. 
It follows that the closed
unit ball of $E$ is
\begin{equation}\label{ball}
	B_E = \bigcup_{a>0} \|a\|^{-1}\,B_{E_a}
\end{equation}
and so the 
norm of a regular $m$-homogeneous polynomial
$P$ is given by
$$
\|P\|= 
\sup\{ \|a\|^{-m} \|P_a\| : a\in E, a>0 \}\,,
$$
where $P_a$ denotes the restriction of $P$
to the Banach lattice $E_a$, equipped with the principal ideal norm.
Applying this to the regular norm, we have
$$
\|P\|_r =
\sup\{ \|a\|^{-m} \||P|_a\| : a\in E, a>0 \}\,.
$$
In the real case, the restriction mapping
$P\in \rhpolys{m}{E}  \mapsto P_a \in \rhpolys{m}{E_a}$
is a lattice homomorphism \cite{BRSNakano}.
It is easy to see that, in the complex case,
the restriction mapping is the complexification
of the real restriction mapping.  Therefore
the restriction mapping is a complex
lattice homomorphism \cite[p.~136]{Sch}.
So we have
$|P|_a = |P_a|$ and it follows that
\begin{align*}
	\|P\|_r &=
	\sup\{ \|a\|^{-m} \||P|_a\| : a\in E, a>0 \}
	= \sup\{ \|a\|^{-m} \||P_a|\| : a\in E, a>0 \}\\
	&=  \sup\{ \|a\|^{-m} \|P_a\|_r : a\in E, a>0 \}\,.
\end{align*}
Thus, from Proposition~\ref{C(K)} we obtain the following result.

\begin{theorem}\label{regnormeq}
  Let $P$ be an orthogonally additive $m$-homogeneous polynomial on a complex
  Banach lattice $E$. Then $\|P\|=\|P\|_r$.
\end{theorem}

In a similar way we can show that if $m$ is an odd integer and $P$ is an
$m$-homogeneous polynomial on a (real) Banach lattice $E$ then $\|P\|=\|P\|_r$,
recovering the first part of \cite[Corollary~1]{BRSGeom}.

From Theorem \ref{regnormeq}
we immediately obtain the 
following result.

\begin{theorem}\label{orthotheo}
  Let $f$ be an orthogonally additive holomorphic function on a complex Banach
  lattice $E$. Then for each $a\in E$ we have
	$$
	|r|(f,a)=r(f,a).
	$$
\end{theorem}

This contrasts with the results of Section 6,
where we have seen that the radii of convergence can be different in general.

\bibliographystyle{amsplain}
\bibliography{Regular}

\noindent Christopher Boyd, School of Mathematics \& Statistics, University
College Dublin, \hfil\break
Belfield, Dublin 4, Ireland.\\
e-mail: christopher.boyd@ucd.ie

\medskip

\noindent Raymond A. Ryan, School of Mathematical and Statistical Sciences, University of Galway, Ireland.\\
e-mail: ray.ryan@universityofgalway.ie

\medskip

\noindent Nina Snigireva, School of Mathematical and Statistical Sciences, University of Galway, Ireland. 
\\
e-mail: nina.snigireva@universityofgalway.ie \\[3pt]

\end{document}